\documentclass[12pt]{amsart}

\usepackage{amssymb}
\usepackage{hyperref}
\usepackage{graphicx}
\usepackage{enumerate}
\usepackage{vmargin}

\newtheorem*{thm*}{Theorem}
\newtheorem{thm}[equation]{Theorem}
\newtheorem{cor}[equation]{Corollary}
\newtheorem{prop}[equation]{Proposition}
\newtheorem{lem}[equation]{Lemma}

\theoremstyle{remark}

\newtheorem*{remark*}{Remark}

\theoremstyle{definition}
\newtheorem{example}[equation]{Example}

\newcommand{\bd}{\mathbf{d}}

\newcommand{\bm}{\mathbf{m}}

\newcommand{\kk}{\Bbbk}
\newcommand{\bbK}{\mathbb{K}}

\newcommand{\dvec}{\bd}
\newcommand{\mvec}{\bm}

\newcommand{\Z}{\mathbb{Z}}

\newcommand{\C}{\mathbb{C}}

\newcommand{\PP}{\mathbb{P}}

\newcommand{\defining}[1]{\textbf{#1}}

\title{Waring decompositions of monomials}

 \author[W.~Buczy\'nska]{Weronika Buczy\'nska}
 \email{wkrych@mimuw.edu.pl}
 \address{Institute of Mathematics of the
 Polish Academy of Sciences\\
 ul. \'Sniadeckich 8\\
 P.O. Box 21\\
 00-956 Warszawa, Poland}

 \author[J.~Buczy\'nski]{Jaros\l{}aw Buczy\'nski}
 \email{jabu@mimuw.edu.pl}
 \address{Institut Fourier \\ Universit\'e Grenoble I \\
 100 rue des Maths, BP 74\\
 38402 St Martin d'H\`eres cedex, France}
 \address{Institute of Mathematics of the
 Polish Academy of Sciences\\
 ul. \'Sniadeckich 8\\
 P.O. Box 21\\
 00-956 Warszawa, Poland}
 \author[Z.~Teitler]{Zach Teitler}
 \email{zteitler@boisestate.edu}
 \address{Department of Mathematics \\ 1910 University Drive \\ Boise State University \\ Boise, ID 83725-1555 \\ USA}

\date{\today}


\dedicatory{To Tony Geramita on the occasion of his 70th birthday.}

\newcommand{\ccI}{{\mathcal{I}}}
\newcommand{\ccJ}{{\mathcal{J}}}
\newcommand{\set}[1]{\left\{#1\right\}}
\newcommand{\fromto}[2]{#1, \dotsc, #2}
\newcommand{\setfromto}[2]{\set{\fromto{#1}{#2}}}
\DeclareMathOperator{\Hilb}{Hilb}
\DeclareMathOperator{\VSP}{VSP}
\newcommand{\VSPaff}{\VSP^{\mathrm{aff}}}
\newcommand{\ccB}{{\mathcal{B}}}
\newcommand{\ccQ}{{\mathcal{Q}}}
\newcommand{\ccS}{{\mathcal{S}}}
\newcommand{\ccT}{{\mathcal{T}}}

\begin{document}

\bibliographystyle{amsalpha}       

\begin{abstract}
A Waring decomposition of a polynomial is an expression of the polynomial as a sum of powers of linear forms,
   where the number of summands is minimal possible.
We prove that any Waring decomposition of a monomial is obtained from a complete intersection ideal,
determine the dimension of the set of Waring decompositions,
and give the conditions under which the Waring decomposition is unique up to scaling the variables.
\end{abstract}

\maketitle

\section{Introduction}

Let $F \in \C[x_0,\dots,x_n]$ be a complex homogeneous polynomial of degree $d$.
The \defining{Waring rank} of $F$, denoted $R(F)$, is the least $r$ such that $F = {\ell_1}^d + \dotsb + {\ell_r}^d$
for some linear forms $\ell_1,\dotsc,\ell_r$.
Any expression $F = {\ell_1}^d + \dotsb + {\ell_r}^d$ with $r = R(F)$ is a \defining{Waring decomposition}.

For example, $xy = \frac{1}{4}(x+y)^2 - \frac{1}{4}(x-y)^2$
and $xyz = \frac{1}{24}(x+y+z)^3 - \frac{1}{24}(x+y-z)^3 - \frac{1}{24}(x-y+z)^3 + \frac{1}{24}(x-y-z)^3$.
Similar decompositions may be found for any monomial (see \textsection\ref{sec: explicit expression}).
These decompositions are not unique, as one may permute the variables and scale them:
for $xy$, replace $(x,y)$ with $(sx,\frac{1}{s} y)$
and for $xyz$, replace $(x,y,z)$ with $(sx,ty,\frac{1}{st}z)$.
For any monomial $F = x_0^{d_0} \dotsm x_n^{d_n}$, one may scale the variables $x_i$ by scalars $\lambda_i$ such that $\prod \lambda_i^{d_i} = 1$,
leaving the monomial fixed but changing the Waring decompositions.
It is natural to ask if all Waring decompositions are obtained in this way;
that is, whether Waring decompositions of monomials are unique up to scaling the variables.

Earlier studies of Waring decompositions have considered the question of uniqueness,
going back to classical results such as Sylvester's Pentahedral Theorem
and more recent work such as \cite{MR1780430,MR2439429},
mostly concentrating on actual uniqueness (not up to scaling) of Waring decompositions of general forms.
More generally, for $F \in \C[\fromto{x_0}{x_n}]$ homogeneous of degree $d$ and $r = R(F)$,
the \textit{variety of sums of powers} $\VSP(F)$ is the closure in $\Hilb_r(\PP^n)$ of the set $\VSP^{\circ}$ of reduced tuples
$\{\fromto{[\ell_1]}{[\ell_r]}\}$ such that $F = \ell_1^d + \dotsb + \ell_r^d$
(see \textsection\ref{sec: dim vsp}).
These varieties have proved to be of interest; see \cite{MR1201387,MR1780430,MR1806733}.

We describe $\VSP(F)$ and determine its dimension when $F = x_0^{d_0} \dotsm x_n^{d_n}$ is a monomial.
We answer the question of uniqueness of Waring decomposition up to scaling the variables,
which amounts to determining whether a torus action on $\VSP^{\circ}(F)$ is transitive.

Both \cite{Ranestad:2011fk} and \cite{Carlini_Catalisano_Geramita}
noted that a Waring decomposition $F = \ell_1^d + \dotsb + \ell_r^d$
may be obtained with $\{\fromto{[\ell_1]}{[\ell_r]}\}$ a complete intersection.
We show that in fact every Waring decomposition of $F$ is a complete intersection of a certain form.
\begin{thm}\label{thm_decomposition_is_a_complete_intersection}
  Suppose $F \in \C[\fromto{x_0}{x_n}]$ is a monomial $F = x_0^{d_0} \dotsm x_n^{d_n}$ with
     $0 < d_0 \leq \dotsb \leq d_n$, $d = d_0 + \dotsb + d_n$, and $F = {\ell_1}^d + \dotsb + {\ell_r}^d$ for $r = R(F)$.
  Let $\ccI \subset \C[\fromto{\alpha_0}{\alpha_n}]$
     be the homogeneous ideal of functions vanishing on $Q = \set{[\ell_1],\dotsc,[\ell_r]} \subset \PP^n$.
  Then $\ccI$ is a complete intersection of degrees $\fromto{d_1+1}{d_n+1}$,
     generated by:
  \[
     \fromto{{\alpha_1}^{d_1 +1} - \phi_1 {\alpha_0}^{d_0+1}}{{\alpha_n}^{d_n +1} - \phi_n {\alpha_0}^{d_0+1}}
  \]
  for some homogeneous polynomials $\phi_i \in \C[\fromto{\alpha_0}{\alpha_n}]$ of degree $d_i - d_0$.
\end{thm}
As a consequence of this and some additional restrictions on the polynomials $\phi_i$
we compute the dimension of the variety of sums of powers of a monomial.
\begin{thm}\label{thm_dim_VSP}
  Suppose $F \in \C[\fromto{x_0}{x_n}]$ is a monomial $F = x_0^{d_0} \dotsm x_n^{d_n}$ with
     $0 < d_0 \leq \dotsb \leq d_n$.
     Let $h$ be the Hilbert function of 
     $\C[\fromto{x_0}{x_n}]/(\fromto{x_1^{d_1+1}}{x_n^{d_n+1}})$.
     Then $\VSP(F)$ is irreducible and $\dim \VSP(F) = h(d_1-d_0) + \dotsb + h(d_n-d_0)$.
\end{thm}
\begin{cor}\label{cor: dim of VSP for equal exponents}
$\dim \VSP(F) \geq n$, with equality if and only if $F = (x_0 \dotsm x_n)^k$.
\end{cor}
Finally we answer the uniqueness question.
\begin{thm}\label{thm_decomposition_uniqueness}
  Suppose $F \in \C[\fromto{x_0}{x_n}]$ is a monomial $F = x_0^{d_0} \dotsm x_n^{d_n}$ with
     $0 < d_0 \leq \dotsb \leq d_n$.
  Let $(\C^*)^{n+1}$ act on $\C[\fromto{x_0}{x_n}]$ by scaling the variables.
  The action of the $n$-dimensional subtorus $\ccT = \{(\fromto{\lambda_0}{\lambda_n}) \mid \prod \lambda_i^{d_i} = 1 \}$
  on $\VSP^{\circ}(F)$ is transitive if and only if $d_0 = \dotsb = d_n$.
\end{thm}

This uniqueness had been shown for $F = xyz$ by Bruce Reznick (2008, personal communication).
Despite the very classical nature of the subject,
it was not possible to address these questions in greater generality until very recently,
as Waring ranks of monomials were only determined in 2011.
This is remarkable when one considers that Waring ranks have been studied for at least 160 years.
Indeed, the ranks and decompositions of quadratic forms were understood classically.
The ranks and decompositions of polynomials in two variables are completely understood,
following work by Sylvester in 1851 \cite{Sylvester:1851kx} and recent work by Comas and Seiguer \cite{MR2754189}.
Beyond these cases it is a difficult problem to determine or even to give decent bounds on $R(F)$,
even for seemingly simple polynomials such as monomials.
For much more on this topic, including history, see \cite{MR1735271} and more recently \cite{Reznick:2010lr}.
Regarding monomials, the paper \cite{Landsberg:2009yq} found $R(xyz) = 4$, $R(xyzw) = 8$, and $R(xyz^2) = 6$.
A recent paper of Ranestad and Schreyer \cite{Ranestad:2011fk} gives a lower bound for rank (of any homogeneous polynomial)
which, for monomials, has the following consequence:
If $F = {x_0}^{d_0} \dotsm {x_n}^{d_n}$ is a monomial,
$d_0 \leq \dotsb \leq d_n$, then
\begin{equation}\label{eq: RS}
  R(F) \geq (d_0 + 1)(d_1 + 1) \dotsm (d_{n-1} + 1) .
\end{equation}
And conversely, it was communicated to us by K.~Ranestad \cite[p.~15]{Ranestad:slides} that
\begin{equation}\label{eq: upperbound}
  R(F) \leq (d_1 + 1) \dotsm (d_{n-1} + 1)(d_n + 1) .
\end{equation}
As a special case, when $d_0= \dotsb =d_n$,
this determines the rank $R((x_0 \dotsm x_n)^d) = (d+1)^n$, see \cite{Ranestad:2011fk}.
The proof of this upper bound described in \cite{Ranestad:2011fk} uses the Bertini theorem;
another proof is given in \cite{Carlini_Catalisano_Geramita}.
We give a third, elementary proof below.

Finally, in 2011 the Waring rank problem for monomials was solved
by Carlini, Catalisano, and Geramita \cite{Carlini_Catalisano_Geramita}.

\begin{thm}[Carlini, Catalisano, Geramita]\label{thm_rank_of_monomials}
The rank of a monomial $x_0^{d_0}\cdots x_n^{d_n}$ with
$d_0\leq\cdots\leq d_n$ is equal to
$(d_1+1)\dotsm(d_n+1)$.
\end{thm}

That is, the rank is equal to the upper bound found by Ranestad and Schreyer.

It is this result of Carlini, Catalisano, and Geramita
which opens the possibility of studying $\VSP$ of monomials.
We give an alternative proof along the way to proving our theorems.
Note, however, that their proof is remarkably direct.

\subsection*{Acknowledgements}

Jaros\l{}aw Buczy\'nski is supported by Maria Sk\l{}odowska-Curie Fellowship ``Contact Manifolds''.

Weronika Buczy\'nska is supported by the grant ``Rangi i rangi brzegowe wielomian\'ow oraz r\'ownania rozmaito\'sci siecznych'',
which is part of the program ``Iuventus Plus'' funded by Polish Ministry of Science and Higher Education.

The authors would like to thank Enrico Carlini and Anthony Geramita for useful comments that helped to improve the exposition.

\subsection*{Notation}

Except in \textsection\ref{sec: explicit expression} we work over $\C$.

We recall a few well-known facts about the apolarity pairing; see, for example, \cite[\textsection 1.1]{MR1735271} for details.
Let $T = \C[x_0,\dotsc,x_n]$ and $S = \C[\alpha_0,\dotsc,\alpha_n]$.
Elements $D \in S$ act on polynomials $F \in T$ by letting $\alpha_i$ act as the differentiation operator $\partial / \partial x_i$.
The induced pairing $S_a \otimes T_d \to T_{d-a}$ (where $T_k = 0$ for $k < 0$) is called the apolarity pairing; it is a perfect pairing when $d=a$.
For a fixed $F \in T$, the set $F^\perp = \{D \in S : DF = 0\}$ is an ideal in $S$, called the annihilator of $F$.
Suppose $F$ is a homogeneous polynomial of degree $d$ and $F = \ell_1^d + \dotsb + \ell_r^d$ for linear forms $\ell_1, \dotsc, \ell_r$.
Consider the set of points $\{[\ell_1], \dotsc, [\ell_r]\}$ in the projective space $\PP^n$.
The defining ideal $\ccI = \ccI(\{[\ell_1], \dotsc, [\ell_r]\})$ of this set of points is a saturated radical homogeneous ideal with $\ccI \subseteq F^\perp$.

\begin{remark*}
     Suppose $F \in \C[\fromto{x_0}{x_n}, \fromto{y_1}{y_m}]$
        is a homogeneous polynomial that depends only on the first $n+1$ variables  $\fromto{x_0}{x_n}$,
        but considered as a polynomial in $n+m+1$ variables.
     If $F = \ell_1^d + \dotsb + \ell_r^d$ is a Waring decomposition (i.e. $r = R(F)$)
        with $\ell_i$ linear forms apriori in $n+m+1$ variables,
        then in fact each $\ell_i$ depends only on $\fromto{x_0}{x_n}$, see \cite[Ex.~3.2.2.2]{Ltensor}.
     Thus one can easily generalize Theorems~\ref{thm_decomposition_is_a_complete_intersection}
        and \ref{thm_dim_VSP} to any monomials, i.e., without the assumption $d_i >0$. 
\end{remark*}

\section{Explicit expression for monomials as sums of powers}\label{sec: explicit expression}

Let $F = x^{\dvec} = {x_0}^{d_0} \dotsm {x_n}^{d_n} \in \kk[x_0,\dotsc,x_n]$ with $0 < d_0 \leq \dotsb \leq d_n$ and let $d = d_0 + \dotsb + d_n$.
For $\kk = \C$, Corollary~3.8 of \cite{Carlini_Catalisano_Geramita} shows that there is a power sum decomposition
\[
  x_0^{d_0} \dotsm x_n^{d_n} = \sum_{\substack{0 \leq a_i \leq d_i \\ i = 1, \dotsc, n}} (x_0 + \zeta_1^{a_1} x_1 + \dotsb + \zeta_n^{a_n} x_n)^d \gamma_{a_1,\dotsc,a_n} ,
\]
where $\zeta_i$ is a primitive $(d_i+1)$-th root of unity for $i = \fromto{1}{n}$, for some coefficients $\gamma$.

In this section, we give the $\gamma$ explicitly.
For this section only, we do not work over $\C$, as we wish to consider the most general field possible.

While we reserve $R(F)$ for the Waring rank of a complex polynomial, one may ask similar questions over other fields.
Given $F \in \kk[x_0,\dotsc,x_n]$, the Waring rank of $F$ with respect to $\kk$, denoted $R_{\kk}(F)$, is the least $r$ such that
$F = c_1 \ell_1^d + \dotsb + c_r \ell_r^d$ for some constants $c_i \in \kk$ and linear forms $\ell_i$ with coefficients in $\kk$.
(When $\kk = \C$, or more generally if $\kk$ is algebraically closed, the coefficients are unnecessary, as $c_i \ell_i^d$ can be replaced by $(c_i^{1/d} \ell_i)^d$.)

In this generality rank may not always be finite, because some polynomials are not sums of powers.
For example, $xy$ is not a sum of squares in characteristic $2$ and $xyz$ is not a sum of cubes in characteristics $2$ or $3$.

If $F$ is a polynomial over $\kk$ and $\kk \subseteq \bbK$ then $R_{\kk}(F) \geq R_{\bbK}(F)$.
In general it may be strictly greater; see, for example, \cite{Oldenburger,Reznick:2010lr}.

Suppose $\kk$ is a field such that $(d_i+1)$-th roots of unity exist in $\kk$ for $i = 1,\dotsc,n$ and $C$ is invertible, where
\[  C = \binom{d}{d_0,d_1,\dotsc,d_n} (d_1+1) \dotsm (d_n+1) , \]
where $\binom{d}{d_0,d_1,\dotsc,d_n}$ is the multinomial coefficient $d! / (d_0! \dotsm d_n!)$.
For $1 \leq i \leq n$ let $\zeta_i \in \kk$ be a primitive $(d_i+1)$-th root of unity. 
We claim that
\begin{equation}\label{eq: monomial expression}
  x_0^{d_0} \dotsm x_n^{d_n} =
      \frac{1}{C}
      \sum_{\substack{0 \leq a_i \leq d_i \\ i = 1, \dotsc, n}}
      (x_0 + \zeta_1^{a_1} x_1 + \dotsb + \zeta_n^{a_n} x_n)^d (\zeta_1^{a_1} \dotsm \zeta_n^{a_n} ) .
\end{equation}
In particular this expression has $(d_1+1) \dotsm (d_n+1)$ summands, implying
$R_{\kk}(F) \leq (d_1 + 1) \dotsm (d_n + 1)$,
which is \eqref{eq: upperbound} when $\kk = \C$.

Consider the monomial $x^{\mvec} = x_0^{m_0} \dotsm x_n^{m_n}$ for some $\mvec = (m_0,\dotsc,m_n)$ such that $m_0 + \dotsb + m_n = d$.
The coefficient of this monomial in the right hand side of the above equation is
\[
  C_{\mvec} = \frac{1}{C} \binom{d}{m_0,\dotsc,m_n} \sum_{\substack{0 \leq a_i \leq d_i \\ i = 1, \dotsc, n}}
 \zeta_1^{a_1(m_1+1)} \dotsm \zeta_n^{a_n(m_n+1)} .
\]
This factors as
\[
  \frac{1}{C} \binom{d}{m_0,\dotsc,m_n} \left( \sum_{a_1 = 0}^{d_1} (\zeta_1^{m_1+1})^{a_1} \right)  \dotsm  \left( \sum_{a_n = 0}^{d_n} (\zeta_n^{m_n+1})^{a_n} \right) .
\]
Since each $\zeta_i^{m_i + 1}$ is still a $(d_i+1)$-th root of unity (not necessarily primitive), we have
\[
  \sum_{a_i = 0}^{d_i} (\zeta_i^{m_i + 1})^{a_i} =
  \begin{cases}
      d_i + 1, & \zeta_i^{m_i+1} = 1, \\
      0 & \text{otherwise}
  \end{cases}
\]
Hence,
\[
  C_{\mvec} =
  \begin{cases}
      \frac{1}{C} \binom{d}{m_0,\dotsc,m_n} (d_1 + 1) \dotsm (d_n + 1), & \zeta_1^{m_1+1} = \dotsb = \zeta_n^{m_n+1} = 1 \\
      0 & \text{otherwise}
  \end{cases}
\]
In particular $C_{\dvec} = 1$, that is, $x^{\dvec} = x_0^{d_0} \dotsm x_n^{d_n}$ appears on the right hand side of \eqref{eq: monomial expression} with coefficient $1$.

Suppose $x^{\mvec}$ has nonzero coefficient in \eqref{eq: monomial expression}.
Then by above $\zeta_i^{m_i + 1} = 1$ for  $i = 1, \dotsc, n$,
   and  we see each $m_i + 1$ is a multiple of $d_i + 1$ for $i > 0$.
Together with $m_0 + m_1 + \dotsb + m_n = d_0 + \dotsb + d_n$ we get $m_0 \leq d_0$.
If for some $i > 0$ we have $m_i + 1 > d_i + 1$, then $m_i + 1 \geq 2(d_i+1)$ and so $m_0 \leq d_0 - (d_i+1) < 0$, since $d_0 \leq \dotsb \leq d_n$.
This contradiction shows that the only term with a nonzero coefficient in \eqref{eq: monomial expression}
is the term having each $m_i + 1 = d_i + 1$, i.e., the term $x^{\dvec}$.

\section{Hilbert function}

Now we begin working toward a proof of Theorem \ref{thm_decomposition_is_a_complete_intersection},
along the way giving an alternative proof of Theorem \ref{thm_rank_of_monomials}.
Henceforth $F = x_0^{d_0} \dotsm x_n^{d_n}$ and $\ccI$ is as described 
  in Theorem~\ref{thm_decomposition_is_a_complete_intersection}.
We start by computing the Hilbert function of $\ccI$.

From this section onwards we work over the field $\kk = \C$.

Let $\ccJ \subset S$ be a complete intersection ideal generated in degrees $d_1+1, \dotsc, d_n+1$.
Note that $\dim(S/\ccJ)_t = (d_1+1)\dotsm(d_n+1)$ for $t \gg 0$.

\begin{prop}\label{prop: Hilbert function agrees}
   Suppose $F = x_0^{d_0} \dotsm x_n^{d_n} = \ell_1^d + \dotsb + \ell_r^d$ with $r$ minimal possible,
      i.e., $r = R(F)$.
   Let $\ccI = \ccI(\{[\ell_1], \dotsc , [\ell_r]\})$.
   Then the Hilbert function of $\ccI$ is equal to the Hilbert function of $\ccJ$.
\end{prop}
The difference between our treatment  of the proposition
and the proof of Theorem~\ref{thm_rank_of_monomials} in \cite{Carlini_Catalisano_Geramita}
is that we compare the Hilbert functions of $\ccI$ and $\ccI \cap (\alpha_0)$,
whereas \cite{Carlini_Catalisano_Geramita} compare $\ccI$ and $\ccI:\alpha_0 + (\alpha_0)$,
which leads to a remarkably quick proof of Theorem~\ref{thm_rank_of_monomials}.
(An earlier version of \cite{Carlini_Catalisano_Geramita} compared $\ccI$ and $\ccI + (\alpha_0)$.)

The remainder of this section gives a proof of this proposition, via some lemmas.
First, a linear algebra lemma will be helpful.
\begin{lem}\label{codim_linear_algebra_fact}
If $A$, $B$, and $C$ are finite dimensional vector spaces such that
\[
\begin{matrix}
  A   & \supset & A\cap B  \\
\cap  &         &   \cap   \\
  C   & \supset &     B    \\
\end{matrix}
\]
then
\[
\dim (A \cap B) \ge \dim A +\dim B -\dim C.
\]
\end{lem}

We denote
\[
  q_t = \dim \Big( \ccI_t      \cap (\alpha_0 \cdot S_{t-1})    \Big).
\]

\begin{lem}\label{lem: bound on qt}
With notation as above, for all $t \in \Z$, $q_t \leq \dim \ccJ_{t-1}.$
Moreover, if for some $t$ we have $q_t < \dim \ccJ_{t-1}$,
   then $q_{t+k d_0} < \dim \ccJ_{t+k d_0-1}$ for all $k > 0$.
\end{lem}

\begin{proof}
For convenience, let $(b_0,b_1,b_2,\dotsc,b_n) = (d_0,d_1+1,d_2+1,\dotsc,d_n+1)$.
Note that
\[
  F^\perp \cap (\alpha_0) = \alpha_0 \cdot (\alpha_0^{b_0},\alpha_1^{b_1},\dotsc,\alpha_n^{b_n}).
\]
It is convenient also to write
\[
  (F^\perp)_t \cap \alpha_0 \cdot S_{t-1} = \alpha_0 \cdot ( \alpha_0^{b_0} S_{t-1-b_0} + \dotsb + \alpha_n^{b_n} S_{t-1-b_n} ) .
\]
Then by inclusion-exclusion we may write the dimension of this space as
\[
  \dim\bigl((F^\perp)_t \cap \alpha_0 \cdot S_{t-1}\bigr) = \sum_{j=1}^{n+1} (-1)^{j-1} \sum_{0 \leq i_1 < \dotsb < i_j \leq n} \dim S_{t-1-b_{i_1} - \dotsb - b_{i_j}} .
\]
Separating terms which omit or include $b_0$,
\[
\begin{split}
  \dim((F^\perp)_t \cap \alpha_0 \cdot S_{t-1})
    &= \left[ \sum_{j=1}^{n} (-1)^{j-1} \sum_{1 \leq i_1 < \dotsb < i_j \leq n} \dim S_{t-1-b_{i_1} - \dotsb - b_{i_j}} \right] \\
    & \quad + \left[ \sum_{j=1}^{n+1} (-1)^{j-1} \sum_{0 = i_1 < \dotsb < i_j \leq n} \dim S_{t-1-b_{i_1} - \dotsb - b_{i_j}} \right] \\
    &= \left[ \sum_{j=1}^{n} (-1)^{j-1} \sum_{1 \leq i_1 < \dotsb < i_j \leq n} \dim S_{t-1-b_{i_1} - \dotsb - b_{i_j}} \right] \\
    & \quad - \left[ \sum_{j=0}^{n} (-1)^{j-1} \sum_{1 \leq i_1 < \dotsb < i_j \leq n} \dim S_{t-b_0-1-b_{i_1} - \dotsb - b_{i_j}} \right] \\
    &= \dim \ccJ_{t-1} - \dim \ccJ_{t-b_0-1} + \dim S_{t-b_0-1} .
\end{split}
\]
Thus
\begin{equation}\label{eq: dimension difference}
  \dim((F^\perp)_t \cap \alpha_0 \cdot S_{t-1}) - \dim S_{t-d_0-1} = \dim \ccJ_{t-1} - \dim \ccJ_{t-d_0-1} .
\end{equation}
We apply this to the diagram below:
\[
\begin{matrix}
\ccI_t \cap \alpha_0 \cdot S_{t-1} & \supset &
\ccI_t \cap ({\alpha_0}^{d_0+1} \cdot S_{t-d_0-1})\\
\cap &&\cap\\
(F^{\perp})_t \cap  (\alpha_0 \cdot S_{t-1}) & \supset &
{\alpha_0}^{d_0+1} \cdot S_{t-d_0-1}. \\
\end{matrix}
\]
By Lemma \ref{codim_linear_algebra_fact},
\[
  q_t - \dim(\ccI_t \cap (\alpha_0^{d_0+1} \cdot S_{t-d_0-1})) \leq \dim \ccJ_{t-1} - \dim \ccJ_{t-d_0-1} .
\]
Note that, since $\ccI$ is radical,
\[
  \ccI_t \cap (\alpha_0^{d_0+1} \cdot S_{t-d_0-1}) = \alpha_0^{d_0} \cdot ( \ccI_{t-d_0} \cap (\alpha_0 \cdot S_{t-d_0-1}) ) ,
\]
which has dimension $q_{t-d_0}$, giving
\begin{equation}\label{eq: q_t difference}
  q_t - q_{t-d_0} \leq \dim \ccJ_{t-1} - \dim \ccJ_{t-d_0-1} .
\end{equation}
Now for any $t$,
\[
\begin{split}
  q_t &= (q_t - q_{t-d_0}) + (q_{t-d_0} - q_{t-2d_0}) + \dotsb \\
    &\leq (\dim \ccJ_{t-1} - \dim \ccJ_{t-d_0-1}) + (\dim \ccJ_{t-d_0-1} - \dim \ccJ_{t-2d_0-1}) + \dotsb \\
    &= \dim \ccJ_{t-1},
\end{split}
\]
as claimed.

If for some $t$ we have $q_t< \dim \ccJ_{t-1}$, then by \eqref{eq: q_t difference} with $t$ replaced by $t+d_0$, we get
   $q_{t+d_0} \leq \dim \ccJ_{t+d_0-1} - \dim \ccJ_{t-1} + q_t < \dim \ccJ_{t+d_0-1}$.
Inductively we obtain the second claim of the lemma.
\end{proof}
We may avoid the inclusion-exclusion and alternating sums in the above proof with the following argument.

\begin{proof}[Alternative proof of Equation \eqref{eq: dimension difference}]
As above, suppose $F = x_0^{d_0} \dotsm x_n^{d_n}$ so that
$F^\perp = (\alpha_0^{d_0+1},\dotsc,\alpha_n^{d_n+1})$.

Let $\ccJ = (\alpha_1^{d_1+1},\dotsc,\alpha_n^{d_n+1})$, so $\ccJ$ is a complete intersection ideal generated in degrees $d_1+1, \dotsc, d_n+1$.

Note that
\[
  F^\perp \cap (\alpha_0) = \alpha_0 \cdot (\alpha_0^{d_0},\alpha_1^{d_1+1},\dotsc,\alpha_n^{d_n+1}) = \alpha_0 \cdot ( \ccJ + (\alpha_0^{d_0}) ) .
\]
Consider the graded $S$-module $\ccJ/\alpha_0^{d_0}\ccJ$.
We write two short exact sequences of graded $S$-modules as follows:
\begin{gather*}
  0 \to \ccJ(-d_0) \xrightarrow{\alpha_0^{d_0}} \ccJ \to \frac{\ccJ}{\alpha_0^{d_0}\ccJ} \to 0 , \\
  0 \to S(-d_0) \xrightarrow{\alpha_0^{d_0}} \ccJ+(\alpha_0^{d_0}) \to \frac{\ccJ}{\alpha_0^{d_0}\ccJ} \to 0 .
\end{gather*}
One may check easily that these are exact (for the second one this is a consequence of $(\alpha_0^{d_0}) \cap \ccJ = \alpha_0^{d_0} \ccJ$).
Counting dimensions of the $(t-1)$-th graded piece,
\[
  \dim \ccJ_{t-1} - \dim \ccJ_{t-d_0-1} = \dim \left(\frac{\ccJ}{\alpha_0^{d_0}\ccJ}\right)_{t-1} = \dim (\ccJ+(\alpha_0^{d_0}))_{t-1} - \dim S_{t-d_0-1} ,
\]
which recovers \eqref{eq: dimension difference} as above and the rest of the proof follows as before.
\end{proof}

\begin{lem}\label{lem: dimension of I vs q_t}
   With notation as above, for all integers $t$ we have
   \[
      \dim (S_{t}/ \ccI_t ) \ge \dim S_{t-1} - q_t.
   \]
\end{lem}
\begin{proof}
   Consider:
   \[
     \begin{matrix}
       \alpha_0 S_{t-1} & \supset & \ccI_t \cap \alpha_0 \cdot S_{t-1}\\
       \cap             &         &\cap\\
       S_{t}            & \supset & \ccI_t.
     \end{matrix}
   \]
   By Lemma~\ref{codim_linear_algebra_fact},
   \[
     \dim (S_{t}/ \ccI_t ) = \dim S_t - \dim \ccI_t \geq \dim S_{t-1} - \dim (\ccI_t \cap \alpha_0 \cdot S_{t-1} ) = \dim S_{t-1} - q_t,
   \]
   which proves the lemma.
\end{proof}

From the two lemmas, we recover the asymptotic version of Proposition~\ref{prop: Hilbert function agrees}.
That is, for sufficiently large $t$, we have:
\begin{align}\label{eq: asymptotic bound}
        R(F) = r &= \dim (S_{t}/ \ccI_t )        && \text{for $t \gg 0$}  \notag \\
        &\ge \dim S_{t-1} - q_t                      && \text{by Lemma~\ref{lem: dimension of I vs q_t}} \\
        &\ge \dim S_{t-1} - \dim \ccJ_{t-1}  && \text{by Lemma~\ref{lem: bound on qt}}  \notag \\
        &= (d_1+1)\dotsm(d_n+1)              && \text{for $t \gg 0$}. \notag
\end{align}
Since $r \le (d_1+1)\dotsm(d_n+1)$ by the explicit expression in Section~\ref{sec: explicit expression},
   all inequalities must be equalities.
This gives an alternative proof of Theorem~\ref{thm_rank_of_monomials}.
In particular, for sufficiently large $t$:
\begin{equation}\label{eq: asymptotic I colon alpha_0 equals I}
   \dim S_t - \dim \ccI_t =  \dim S_{t+1} - \dim \ccI_{t+1} = \dim S_{t} - q_{t+1}, \text{ i.e., } \dim \ccI_t = q_{t+1}.
\end{equation}

\begin{lem}\label{lem: I colon alpha_0 equals I}
  We have $\ccI:\alpha_0 = \ccI$.
\end{lem}

\begin{proof}
Multiplication by $\alpha_0$ gives a one-to-one map $\ccI_t \to \ccI_{t+1} \cap \alpha_0 \cdot S_t$.
For $t \gg 0$, since $\dim \ccI_t = q_{t+1}$, this multiplication map is onto, hence $(\ccI: \alpha_0)_t = \ccI_t$ in sufficiently high degree.
Thus $\ccI$ and $\ccI:\alpha_0$ agree up to saturation.
But both ideals are saturated, so $\ccI = \ccI:\alpha_0$.

Explicitly, let $\beta \in (\ccI:\alpha_0)_t$.
Then $\beta^N \in (\ccI:\alpha_0)_{tN}$ and $\alpha_0 \beta^N \in \ccI_{tN+1}$.
For $N \gg 0$, $\dim \ccI_{tN} = \dim (\ccI \cap (\alpha_0))_{tN+1}$, so $\beta^N \in \ccI_{tN}$, and since $\ccI$ is radical, $\beta \in \ccI$.
\end{proof}

\begin{lem}\label{lem: dim I_t equals q_{t+1}}
   For all $t$, we have $\dim \ccI_t = q_{t+1}$.
\end{lem}

\begin{proof}
Lemma \ref{lem: I colon alpha_0 equals I} shows that multiplication by $\alpha_0$
gives a bijection $\ccI_t = (\ccI:\alpha_0)_t \to (\ccI \cap (\alpha_0))_{t+1}$ in every degree $t$.
\end{proof}

\begin{proof}[Proof of Proposition~\ref{prop: Hilbert function agrees}]
   By Lemmas~\ref{lem: bound on qt} and \ref{lem: dim I_t equals q_{t+1}},
   \[
      \dim \ccI_t    = q_{t+1}  \le  \dim \ccJ_t.
   \]
   By the ``moreover'' part of Lemma~\ref{lem: bound on qt},
      if for some $t$ we have a strict inequality $\dim  \ccI_t < \dim \ccJ_t$,
      then for all nonnegative $k$,
         \[
            \dim (S_{t+k d_0}/ \ccI_{t+k d_0} ) > \dim S_{t+k d_0} - \dim \ccJ_{t+k d_0},
         \]
      a contradiction with \eqref{eq: asymptotic bound} for sufficiently large $k$.
\end{proof}

\begin{cor}
For $F = x_0^{d_0} \dotsm x_n^{d_n} = \ell_1^d + \dotsb + \ell_r^d$ with $d_0 \leq \dotsb \leq d_n$ and $r=R(F)$,
each of the linear forms $\ell_i$ has $x_0$ appearing with nonzero coefficient.
\end{cor}
\begin{proof}
The coordinate $\alpha_0$ on a point $[\ell] = [\alpha_0 x_0 + \dotsb + \alpha_n x_n] \in \PP T_1$ gives the coefficient of $x_0$ in the linear form $\ell$.
As above, let $\ccI$ be the defining ideal of the set of points $Q = \set{[\ell_1],\dotsc,[\ell_r]}$.
Then $\ccI:\alpha_0$ is the defining ideal of the subset of points in $Q$ not lying on the hyperplane $\{\alpha_0=0\}$.
By Lemma \ref{lem: I colon alpha_0 equals I}, $\ccI:\alpha_0 = \ccI$, so all the points $[\ell_i]$ 
  are not contained in this hyperplane, and hence have nonzero coefficients of $x_0$.
\end{proof}

\section{Complete intersections}

Fix $k \in \setfromto{0}{n}$, and $\bar \phi  = (\fromto{\phi_1}{\phi_k})$, with $\phi_i \in S_{d_i- d_0}$.
We define the following homogeneous ideal:
\[
  \ccI(k, \bar \phi) = ( \alpha_i^{d_i+1} - {\alpha_0}^{d_0+1} \phi_i \mid i \in \setfromto{1}{k} ).
\]
Every such ideal is a complete intersection ideal generated in degrees $\fromto{d_1+1}{d_k+1}$.

\begin{example}
The explicit expression in Section~\ref{sec: explicit expression} corresponds to
\[
   \bar\phi = (\fromto{\alpha_0^{d_1-d_0}}{\alpha_0^{d_n-d_0}}).
\]
\end{example}

The following Proposition proves Theorem~\ref{thm_decomposition_is_a_complete_intersection}.
\begin{prop}\label{prop:any_solution_is_ccI_n_bar_phi}
   Any ideal $\ccI$ as in Theorem~\ref{thm_decomposition_is_a_complete_intersection}
      is of the form $\ccI(n, \bar \phi)$, for some $\bar \phi  = (\fromto{\phi_1}{\phi_n})$.
\end{prop}

\begin{proof}
   Let $\ccI_{\le t}$ denote the ideal generated by the homogeneous elements in $\ccI$ of degree at most $t$.
   We will prove by induction that $\ccI_{\le t} = \ccI(k, \bar \phi)$ for some $k=k(t)$ and $\bar \phi= \bar \phi(t)$.
   More precisely, we claim that $k(t)=\# \set{i :  d_i+1 \le t}$.
   For $t=0$, $\ccI_{\le 0}=\ccI(0, \emptyset) = 0$.
   Suppose $\ccI_{\le t-1} = \ccI(k, \bar \phi)$ for $k = \# \set{i :  d_i+1 \le t-1}$ and some $\bar \phi= \bar \phi(t-1)$.
   If $t\ne d_i+1$ for any $i \in \setfromto{1}{n}$, then
     \[
       \dim \ccI_t = \dim \ccJ_t = \dim (\ccI_{\le t-1})_{t}.
     \]
   The first equality follows from Proposition~\ref{prop: Hilbert function agrees}
      and the second from the fact that $\ccI(k, \bar \phi)$ is a complete intersection
      of the same degrees as $\ccJ$, up to degree $t$.
   So there is no new generator of $\ccI$ in degree $t$,
      and $\ccI_{\le t} = \ccI(k, \bar \phi)$.

   Now suppose $t =d_{k+1}+1 = \dotsb = d_l+1 < d_{l+1} + 1$ for some $l > k$.
   Then by the same argument, there must be exactly $l-k$
      new linearly independent generators $(\fromto{\rho_{k+1}}{\rho_l})$ of $\ccI$ in degree $t$.
   Each $\rho_i$ must be in $(F^{\perp})_t$.
   Using the generators of the lower degree, we may eliminate from $\rho_i$
      the summands divisible by $\alpha_j^{d_j+1}$ for each $1 \le j \le k$.
   That is, we can assume that each of the new generators is of the form
   \[
     \rho_i =  \sum_{j = k+1}^l c_{ij} {\alpha_j}^{d_j+1} -  \psi_i {\alpha_0}^{d_0+1}
   \]
   for some $\psi_i \in S_{t-d_0-1}$  and $c_{ij} \in \C$.
   We claim the matrix $C :=(c_{ij})_{i,j=k+1}^l$ is invertible.
   Suppose on contrary that there exists a linear combination $\rho$ of the $\rho_i$'s
      such that $\rho = \psi {\alpha_0}^{d_0+1}$ (possibly $\psi = 0$).
   Since $\ccI$ is radical, $\psi \alpha_0 \in \ccI_{t-d_0}$.
   But then $\rho \in \ccI_{\le t-1}$ and one of the generators $\rho_i$ is redundant, a contradiction.

   So $C$ is invertible, and by replacing $(\rho_{k+1},\dotsc,\rho_l)$ with the linear combinations
$C^{-1}(\rho_{k+1},\dotsc,\rho_l)^t$
     (and analogously for $\psi_i$),
     we may assume:
   \[
     \rho_i = {\alpha_i}^{d_i+1} -  \psi_i {\alpha_0}^{d_0+1}.
   \]

   Set $\bar \phi'  = (\fromto{\phi_1}{\phi_k}, \fromto{\psi_{k+1}}{\psi_{l}})$.
   By the above, we have $\ccI_{\le t} = \ccI(l, \bar \phi')$.
\end{proof}

We have just seen that if $\ccI \subset F^\perp$ is a radical one-dimensional ideal, then $\ccI = \ccI(n,\bar\phi)$ for some $\bar\phi$.
One may ask which $\bar \phi = (\phi_1,\dotsc,\phi_n)$ can occur.
For any $\bar\phi$, $\ccI(n,\bar\phi)$ is a one-dimensional complete intersection ideal.
So the question is, for which $\bar\phi$ is $\ccI(n,\bar\phi)$ radical?
An obvious necessary condition is that each $\phi_i$ should not be divisible by $\alpha_i^2$.

In addition, if $\ccI = \ccI(n,\bar\phi)$ is radical, then each $\phi_i \notin \ccI(i-1,(\phi_1,\dotsc,\phi_{i-1}))$, the subideal generated by the first $i-1$ generators of $\ccI$.
Otherwise the generator $\alpha_i^{d_i+1} - \phi_i \alpha_0^{d_0+1}$ of $\ccI$
can be replaced with $\alpha_i^{d_i+1}$, so the ideal is not radical.
The following example shows that this necessary condition is not sufficient, even combined with $\alpha_i^2 \nmid \phi_i$.

\begin{example}
For $F = xy^2z^3$,
$\ccI$ must have the form $\ccI = (\beta^3-L\alpha^2,\gamma^4-Q\alpha^2)$ for some linear form $L=L(\alpha,\beta,\gamma)$ and quadratic form $Q=Q(\alpha,\beta,\gamma)$.
Consider the example $\ccI = (\beta^3 - \alpha^2 \gamma, \gamma^4 - \alpha^2 \beta^2)$.
Then $L = \gamma$ is not divisible by $\beta^2$, $Q = \beta^2$ is not divisible by $\gamma^2$, and $Q \notin (\beta^3 - \alpha^2 \gamma)$,
   nevertheless this ideal $\ccI$ is not radical.
One can check easily that $P = \alpha^4 \beta - \beta^2 \gamma^3 \notin \ccI$ but $P^2 \in \ccI$.
More concretely, $\ccI$ defines a scheme of length $2$ at $[1:0:0] \in \PP^2$.
\end{example}

\begin{prop}\label{prop:general_phi_bar_gives_radical}
   Let $\bar\phi$ be general, i.e., each $\phi_i \in S_{d_i-d_0}$ is general.
   Then $\ccI(n,\bar\phi)$ is radical.
\end{prop}

\begin{proof}
   Let $\ccB:= \prod_{i =1}^{n} S_{d_i-d_0}$.
   In $\ccB \times \PP^n$ consider a variety $\ccQ$ defined by
     $\ccI(n,\bar\phi)$,
     where $\bar \phi  = (\fromto{\phi_1}{\phi_n})$,
     $\ccI(n, \bar \phi) = ( \alpha_i^{d_i+1} - {\alpha_0}^{d_0+1} \phi_i \mid i \in \setfromto{1}{n})$,
     and $\phi_i =  \sum_{|J| = d_i-d_0} f_J \alpha^J$, where $f_J$ are the coordinates on the $S_{d_i - d_0}$
     component of $\ccB$.
   Then $\ccQ$ is a complete intersection and each fiber $\ccQ_b$ of $\ccQ \to \ccB$ is a complete intersection.
   In particular, $\ccQ$ is Cohen-Macauley \cite[Prop.~18.13]{eisenbud}
     and the map $\ccQ \to \ccB$ is equidimensional, thus flat \cite[Thm~18.16]{eisenbud}.
   Some of the fibers are reduced, and since being reduced is an open property in a flat family,
     it follows that a general fiber is reduced.
\end{proof}

\section{Dimension of variety of sums of powers}\label{sec: dim vsp}

For a homogeneous form $F \in T$ of degree $d$, and $r > 0$,
let
\[
  \VSP^{\mathrm{aff},\circ}(F,r) = \left\{ \hat{Q} = \{\fromto{\ell_1}{\ell_r}\} \, \mid \, \text{$\hat{Q} \in \Hilb_r(T_1)$ is reduced and $F = \ell_1^d + \dotsb + \ell_r^d$} \right\}
\]
and
\[
  \VSPaff(F,r) = \overline{\VSP^{\mathrm{aff},\circ}(F,r)} \subset \Hilb_r(T_1) .
\]
Also, let
\begin{multline*}
  \VSP^{\circ}(F,r) = \left\{ Q = \{\fromto{[\ell_1]}{[\ell_r]}\} \, \right.\mid \\ 
         \left. \text{$Q \in \Hilb_r(\PP(T_1))$ is reduced and $F = \ell_1^d + \dotsb + \ell_r^d$} \right\}
\end{multline*}
and
\[
  \VSP(F,r) = \overline{\VSP^{\circ}(F,r)} \subset \Hilb_r(\PP(T_1)) .
\]
When $r = R(F)$ we omit it from the notation.
$ \VSP(F,r)$ is called the \emph{variety of sums of powers}, see \cite{MR1780430}.
$\VSP^{\mathrm{aff},\circ}(F)$ parametrizes Waring decompositions of $F$.
We make the general remark that $\dim \VSPaff(F) = \dim \VSP(F)$.
First, observe the following easy lemma.
\begin{lem}\label{lem:unique_solution_to_linear_algebra}
   Suppose $F = {\ell_1}^d + \dotsb + {\ell_r}^d$ is a Waring decomposition of a homogeneous polynomial $F$
     (i.e., $r= R(F)$).
   If $F = c_1{\ell_1}^d + \dotsb + c_r{\ell_r}^d$ for some $c_i \in \C$, then
     $c_1 =\dotsb = c_r =1$.
\end{lem}

\begin{proof}
   Suppose on the contrary that for instance $c_r \ne 1$.
   Then combining the two decompositions we obtain:
   \begin{align*}
     (c_r -1) F  & = c_r({\ell_1}^d + \dotsb + {\ell_r}^d)  - (c_1{\ell_1}^d + \dotsb + c_r{\ell_r}^d)\\
                 & = (c_r-c_1){\ell_1}^d + \dotsb + (c_r-c_{r-1}){\ell_{r-1}}^d.
   \end{align*}
   This contradicts the minimality of $r$.
\end{proof}

By this lemma, the obvious projectivization map $\VSP^{\mathrm{aff},\circ}(F) \to \VSP^{\circ}(F)$ is finite of degree $d^r$ ($r = R(F)$).
Indeed, each $\ell_i$ can only be replaced by a scalar multiple $\lambda \ell_i$ when $\lambda$ is a $d$-th root of unity.

We now turn to monomials.
As before, let $F = x_0^{d_0} \dotsm x_n^{d_n}$, $0 < d_0 \leq \dotsb \leq d_n$, $d = d_0 + \dotsb + d_n$, and $r = R(F) = (d_1+1)\dotsm(d_n+1)$.
Also, let $\ccJ = (\fromto{\alpha_1^{d_1+1}}{\alpha_n^{d_n+1}})$.
We get a formula for the dimension of the space of solutions to the Waring decomposition problem.

\begin{prop}\label{prop:dimension_of_VSP}
   Suppose $F$ is a monomial as above.
   Let $h$ be the Hilbert function of $S/\ccJ$.
   Then $\VSP(F)$ is irreducible and $\dim \VSP(F) = h(d_1-d_0) + h(d_2-d_0) + \dotsb + h(d_n-d_0)$.
\end{prop}

To prove this, we first describe the space parametrizing the radical one-dimensional ideals $\ccI \subset F^\perp$,
corresponding to points in $\VSP^{\circ}(F)$.

\begin{prop}\label{prop: unique generators}
Any ideal $\ccI$ as in Proposition~\ref{prop: Hilbert function agrees} is of the form $\ccI(n,\bar\phi)$
for a unique $\bar\phi = (\fromto{\phi_1}{\phi_n})$ such that no term of any $\phi_i$ lies in $\ccJ$.
\end{prop}

\begin{proof}
Existence of such a $\bar\phi$ is obvious: if a term of $\phi_i$ is divisible by some $\alpha_j^{d_j+1}$, it can be eliminated by subtracting an appropriate multiple of
the generator $\alpha_j^{d_j+1} - \phi_j \alpha_0^{d_0+1}$ of $\ccI$.
Suppose $\phi_i$ can be replaced by $\phi'_i$, leaving $\ccI$ the same, with no term of either $\phi_i$ or $\phi'_i$ lying in $\ccJ$.
Then the generators $\alpha_i^{d_i+1} - \phi_i \alpha_0^{d_0+1}$ and $\alpha_i^{d_i+1} - \phi'_i \alpha_0^{d_0+1}$ must differ by a combination of previous generators,
\begin{equation}\label{eq: changing generator}
  (\phi_i - \phi'_i) \alpha_0^{d_0+1} = \sum_{j=1}^{i-1} \psi_j (\alpha_j^{d_j+1} - \phi_j \alpha_0^{d_0+1}) ,
\end{equation}
for some $\psi_j$.
The right hand side is a combination of generators of $\ccI$ with no term in $\ccJ$.
From the above equation it is clearly divisible by $\alpha_0^{d_0+1}$.
We claim more generally that any element of $\ccI$ with no term in $\ccJ$
is necessarily divisible by $\alpha_0^{d_0+1}$.
Indeed, in a combination as in \eqref{eq: changing generator}, each term of each product $\psi_j \alpha_j^{d_j+1}$ must cancel;
the only surviving terms arise from the products $\psi_j \phi_j \alpha_0^{d_0+1}$.

Now, suppose $\beta \in \ccI$ has no term in $\ccJ$, so $\beta$ is divisible by $\alpha_0^{d_0+1}$.
Then $\gamma = \beta / \alpha_0^{d_0+1}$ lies in $\ccI : \alpha_0^{d_0+1}$, which is $\ccI$, by Lemma~\ref{lem: I colon alpha_0 equals I}.
And it is still true that no term of $\gamma$ is in $\ccJ$, so again $\gamma$ is divisible by $\alpha_0^{d_0+1}$.
Continuing in this way, $\beta$ is divisible by $(\alpha_0^{d_0+1})^k$ for every $k \geq 0$.
Hence $\beta = 0$.

Applying this to \eqref{eq: changing generator}, $\phi_i - \phi'_i = 0$, as desired.
\end{proof}

We have the following converse.

\begin{prop}\label{prop: general restricted bar phi give radical}
Let $\bar\phi = (\fromto{\phi_1}{\phi_n})$ be a general tuple of homogeneous polynomials of degree $\deg \phi_i = d_i - d_0$
with no terms in $\ccJ$.
Then $\ccI(n,\bar\phi)$ is radical.
\end{prop}
\begin{proof}
For each $i$, let $B'_i \subset S_{d_i - d_0}$ be the set of monomials of degree $d_i - d_0$ which are not in $\ccJ$,
so $B'_i$ gives a basis for $(S/\ccJ)_{d_i - d_0}$.
For each $i$, let $\langle B'_i \rangle \subset S_{d_i - d_0}$ be the linear span of $B'_i$.
The hypothesis means each $\phi_i$ is general in $\langle B'_i \rangle$.

Let $\ccB' = \prod_{i=1}^n \langle B'_i \rangle$.
Let $\ccQ' \subset \ccB' \times \PP^n$ be defined by
$\ccI(n,\bar\phi) = (\alpha_i^{d_i+1} - \phi_i \alpha_0^{d_0+1} \mid i \in \{1,\dotsc,n\})$ where
$\bar\phi \in \ccB'$ and $(\fromto{\alpha_0}{\alpha_n})$ are coordinates on $\PP^n$.
The rest of the proof is the same as the proof of Proposition~\ref{prop:general_phi_bar_gives_radical}.
For the statement that ``some of the fibers are reduced'', a reduced fiber is provided by the explicit expression
given in Section~\ref{sec: explicit expression}.
\end{proof}

Now the proof of Proposition~\ref{prop:dimension_of_VSP} is immediate.
\begin{proof}[Proof of Proposition~\ref{prop:dimension_of_VSP}]
The map $\ccB' \dashrightarrow \VSP^{\circ}(F)$, given by $\bar\phi \mapsto \ccI(n,\bar\phi)$
(the defining ideal of the fiber $\ccQ'_{\bar\phi} \subset \PP^n$) whenever this is radical,
is defined on an open subset of $\ccB'$ by Proposition~\ref{prop: general restricted bar phi give radical},
and (on this open subset) is one-to-one and onto by Proposition~\ref{prop: unique generators}.
Since $\ccB'$ is irreducible, so is $\VSP(F)$, and
\[
  \dim \VSP(F) = \dim \VSP^{\circ}(F) = \dim \ccB' = \sum \dim \langle B'_i \rangle = \sum h(d_i-d_0)
\]
as claimed.
\end{proof}

\begin{proof}[Proof of Corollary \ref{cor: dim of VSP for equal exponents}]
Since $d_i \geq d_0$, each $h(d_i-d_0) \geq 1$.
Since $\ccJ$ has no generators in degree $1$, $h(1) = \dim S_1 \geq 2$, and $h(t) \geq h(1)$ for $t \geq 1$,
so $h(d_i-d_0) = 1$ if and only if $d_i - d_0 = 0$.
\end{proof}


\begin{example}
For $F = x^2 y^2 z^2$,
$\ccI$ must have the form $\ccI = (\beta^3 - a \alpha^3, \gamma^3 - b \alpha^3)$ for some constants $a,b$,
so $\dim \VSP(F) = 2$, and this is equal to $h(d_1-d_0) + h(d_2-d_0) = 2 h(0)$.

This ideal is a radical complete intersection if and only if $ab \neq 0$.
In that case, rescaling the variables takes $\ccI$ to $(\beta^3-\alpha^3,\gamma^3-\alpha^3)$.
\end{example}

Finally, we consider the natural group operation on $\VSP(F)$ mentioned in the introduction.
Let $(\C^*)^{n+1}$ act on $S = \C[x_0,\dotsc,x_n]$ by scaling the variables.
The $n$-dimensional subtorus $\ccT = \{ (\fromto{\lambda_0}{\lambda_n}) \, \mid \, \prod \lambda_i^{d_i} = 1 \}$
leaves $F = x_0^{d_0} \dotsm x_n^{d_n}$ fixed, and so acts on $\VSP(F)$.
Also, let $\ccS_F$ be the group of permutations of $\{0,\dotsc,n\}$ respecting the degree tuple $(d_0,\dotsc,d_n)$,
in the sense that a permutation $\pi$ lies in $\ccS_F$ if and only if, for each $i$, $d_{\pi(i)} = d_i$.
When $\ccS_F$ acts on $S$ by permuting variables it leaves $F$ fixed, so again acts on $\VSP(F)$.
The actions of $\ccT$ and $\ccS_F$ commute, so $\VSP(F)$ carries an action of the $n$-dimensional algebraic group $\ccT \times \ccS_F$.

Theorem~\ref{thm_decomposition_uniqueness} states that
this action is transitive (in fact the $\ccT$-action alone is already transitive) if and only if $F = (x_0\dotsm x_n)^k$.
\begin{proof}[Proof of Theorem~\ref{thm_decomposition_uniqueness}]
Let $F = x_0^{d_0} \dotsm x_n^{d_n}$ with $d_0 \leq \dotsb \leq d_n$.
If $d_0 < d_n$ then $\dim \VSP^{\circ}(F) > n = \dim \ccT = \dim (\ccS_F \times \ccT)$, so the action can not be transitive.
If $d_0 = \dotsb = d_n = k$ then by Theorem~\ref{thm_decomposition_is_a_complete_intersection}
every Waring decomposition of $F$ is cut out by an ideal of the form
$(\fromto{\alpha_1^{k+1} - \phi_1 \alpha_0^{k+1}}{\alpha_n^{k+1} - \phi_n \alpha_0^{k+1}})$,
where the $\phi_i$ are scalars.
Since the ideal of a Waring decomposition is a radical ideal, the $\phi_i$ are in fact nonzero scalars.
Then scaling each $x_i$ by $\phi_i^{-1/k}$ takes this Waring decomposition to the one cut out by
$(\fromto{\alpha_1^{k+1} - \alpha_0^{k+1}}{\alpha_n^{k+1} - \alpha_0^{k+1}})$,
showing that $\ccT$ acts transitively on $\VSP^{\circ}(F)$.
\end{proof}

\providecommand{\bysame}{\leavevmode\hbox to3em{\hrulefill}\thinspace}
\providecommand{\MR}{\relax\ifhmode\unskip\space\fi MR }
\providecommand{\MRhref}[2]{%
  \href{http://www.ams.org/mathscinet-getitem?mr=#1}{#2}
}
\providecommand{\href}[2]{#2}

\end{document}